\documentclass[twoside]{amsart}
\usepackage[dvips]{graphicx}   
\usepackage{color,epsfig}      

\usepackage{amsmath,amssymb,enumerate}
\usepackage[margin=2.5cm,nohead]{geometry}
\usepackage[utf8]{inputenc}
\usepackage[english]{babel}
\usepackage{latexsym}
\usepackage{amsfonts}
\usepackage{amssymb}
\usepackage{amsthm}
\usepackage{setspace}
\usepackage{graphicx}
\usepackage{epstopdf}
\usepackage{caption}
\usepackage{lmodern}
\usepackage{indentfirst}
\usepackage{geometry}
\usepackage{faktor}
\usepackage{diagrams}

\newtheorem{thm}{Theorem}[section]

\newtheorem{lemma}[thm]{Lemma}

\newtheorem{cor}[thm]{Corollary}
\newtheorem{prop}[thm]{Proposition}
\newtheorem{defn}[thm]{Definition}
\newtheorem{rem}[thm]{Remark}

\newcommand{\rect}{\mathop{\mathrm{Rect}^{\circ}}}
\newcommand{\Rect}{\mathop{\mathrm{Rect}}}
\newcommand{\pent}{\mathop{\mathrm{Pent}^{\circ}}}
\newcommand{\Pent}{\mathop{\mathrm{Pent}}}
\newcommand{\hex}{\mathop{\mathrm{Hex}^{\circ}}}
\newcommand{\Hex}{\mathop{\mathrm{Hex}}}
\newcommand{\Ker}{\mathop{\mathrm{Ker}}}
\renewcommand{\Im}{\mathop{\mathrm{Im}}}
\newcommand{\XX}{\mathbb{X}}
\newcommand{\OO}{\mathbb{O}}
\newcommand{\NN}{\mathbb{N}}
\newcommand{\ZZ}{\mathbb{Z}}
\newcommand{\RR}{\mathbb{R}}

\newcommand{\gr}{\mathop{\mathrm{gr}_t}}
\newcommand{\gs}[1]{\mathbf{#1}}
\newcommand{\Int}{\mathrm{Int}}
\newcommand{\SG}{\mathbf{S}(\mathbb{G})}
\newcommand{\cR}{\mathcal{R}}

\newcommand{\GG}{\mathbb{G}}
\newcommand{\Xbar}{\bar{\mathbb{X}}}
\newcommand{\Obar}{\bar{\mathbb{O}}}
\newcommand{\Tors}{\mathop{\mathrm{Tors}}}
\newcommand{\Cone}{\mathop{\mathrm{Cone}}}
\newcommand{\shift}[1]{\left[\!\left[#1\right]\!\right]}

\newarrow{Corresponds}<--->

\begin{document}

\title[] {The Knot Invariant $\Upsilon$ Using Grid Homologies}

\author[Vikt\'oria F\"oldv\'ari]{Vikt\'oria F\"oldv\'ari}

\address{Vikt\'oria F\"oldv\'ari,
 Institute of Mathematics, E\"otv\"os Lor\'and University, Budapest, Hungary
}
\email{foldvari@math.elte.hu}

                                                                              
\begin{abstract}
According to the idea of Ozsváth, Stipsicz and Szabó, we define the knot invariant $\Upsilon$ without the holomorphic theory, using constructions from grid homology. We develop a homology theory using grid diagrams, and show that $\Upsilon$, as introduced this way, is a well-defined knot invariant. We reprove some important propositions using the new techniques, and show that $\Upsilon$ provides a lower bound on the unknotting number.
\end{abstract}

\maketitle
                                                                                

\section{Introduction}\label{inadvance}

In 2014, Ozsváth, Stipsicz and Szabó \cite{ozsvath:2014} introduced the knot invariant $\Upsilon$ using knot Floer homology. In their book \cite{ozsvath2015grid} they used grid homologies to develop the concordance invariant $\tau$ first defined in \cite{MR2026543}, and among other open problems they presumed that $\Upsilon$ can also be defined without the holomorphic theory. This idea is reasonable and important: for example the fact that $\upsilon(K)=\Upsilon_K(1)$ gives a lower bound on the unoriented four-ball genus \cite{unoriented}, have only been proved using methods based on grid diagrams. Our aim is to work out this new approach: to develop a homology theory using grid diagrams of knots, and to introduce the knot invariant $\Upsilon$ using grid homologies. The significance of this new definition is that it gives the opportunity to examine problems combinatorically. We reprove some important properties using the new techniques.

Former versions of grid diagrams have been studied for more than a century as convenient tools for understanding knots and links \cite{brunn1897, Cromwell199537, MR2232855, MR555835, MR792695}. Since every oriented link can be represented by a grid diagram (see \cite[Chapter 3.]{ozsvath2015grid} for the basic definitions), this theory has the advantage of handling knots and links with discrete methods. In the second half of the twentieth century, knot theory was connected with more and more different areas of mathematics, which led to outstanding results. 

In recent years, knot theory has undergone significant progress, mostly because of the incorporation of algebraic concepts, such as knot Floer homology, Heegaard Floer homology, Khovanov homology and the generalizations of knot polynomials. Grid diagrams have also been reconsidered \cite{MR2480614, MR2500576, MR2915478} to build up grid homology theories which enabled us to approach problems algebraically.

The main points of this paper are the new introduction of $\Upsilon$ in Definition \ref{def:Upsilon}, and Theorem \ref{upsiloninvariantundergridmove} stating that $\Upsilon$, as defined using grids, is a well-defined knot invariant. In Corollary \ref{cor:unknottingbound} we give a lower bound on the unknotting number. Note that the results were known earlier \cite{ozsvath:2014} but their proof was based on the holomorphic theory.

Before proceeding further, we state here some properties of $\Upsilon_K: [0,2]\rightarrow \mathbb{R}$ that are proved in \cite{ozsvath:2014}.\\

\begin{itemize}
\item $\Upsilon_K$ is additive under connected sum of knots:
$$\Upsilon_{K_1 \# K_2}(t)=\Upsilon_{K_1}(t)+\Upsilon_{K_2}(t).$$
\item For the mirror image $m(K)$ of the knot $K$, $\Upsilon_{m(K)}(t)=-\Upsilon_K(t)$.
\item $\Upsilon_K(t)=\Upsilon_K(2-t).$
\item Suppose that $K_1$ and $K_2$ are two knots that can be connected by a genus $g$ cobordism in $[0,1]\times S^3$. Then for $t\in [0,1]$
$$|\Upsilon_{K_1}(t)-\Upsilon_{K_2}(t)|\leq t\cdot g.$$
\item For a knot $K$ and $t\in [0,1]$ the invariant $\Upsilon_K(t)$ bounds the slice genus $g_s(K)$ of $K$:
$$|\Upsilon_K(t)|\leq t\cdot g_s(K).$$
\item For an alternating knot $K$, the $\Upsilon_K$ invariant can be explicitly computed:
$$\Upsilon_K(t)=(1-|t-1|)\cdot \frac{\sigma}{2},$$
where $\sigma$ is the signature of the knot $K$.
\item $\Upsilon_K(0)=0$ and $\Upsilon_K(2)=0$.
\item For $m, n \in \NN_{> 0}$, the quantity $\Upsilon_K(\frac{m}{n})$ lies in $\frac{1}{n} \ZZ$.
\item The function $\Upsilon_K: t\mapsto \Upsilon_K(t)$ is continuous and piecewise linear. Each slope is an integer.
\item The slope of $\Upsilon_K(t)$ at $t=0$ is given by $-\tau(K)$, where $\tau$ is a knot invariant, defined in Chapter 6 of \cite{ozsvath2015grid}.
\item If the knots $K_1$ and $K_2$ can be connected with a concordance, then $\Upsilon_{K_1}(t)=\Upsilon_{K_2}(t)$ for $t\in[0,2]$. Furthermore, $[K]\mapsto \Upsilon_K$ is a group homomorphism from the concordance group to the space $C=C^0[0,2]$ of continuous functions on $[0,2]$.

\end{itemize}

For torus knots, the Alexander polynomial \cite{MR1501429} determines $\Upsilon_K(t)$, see \cite{ozsvath:2014}.

$\mathbf{Acknowledgement:}$ I am grateful to my supervisor, András Stipsicz, for his help and corrections.

\section{Background}\label{gridchapter}

\subsection{Planar and toric grid diagrams}

A \emph{planar grid diagram} $\GG$ with \emph{grid number} $n$ is a square on the plane with $n$ rows and $n$ columns of small squares (i.e. an $n\times n$ grid), marked with $X$'s and $O$'s in a way that no square contains both $X$ and $O$, and each row and each column contains exactly one $X$ and one $O$.

We use the notation $\XX$ for the set of squares marked with an $X$, and $\OO$ for the ones containing an $O$. Sometimes we label the markings: $\{O_i\}_{i=1}^n$, $\{X_i\}_{i=1}^n$.

We usually work with a \emph{toroidal grid diagram}, which can be obtained by identifying the opposite sides of a planar grid diagram: its top boundary segment with its bottom one and its left boundary segment with its right one. A \emph{planar realization} of a toroidal grid diagram is a planar grid diagram obtained by cutting up the toroidal grid diagram along a horizontal and a vertical line.

Every grid diagram $\GG$ determines a diagram of an oriented link in the following way:
In each row connect the $O$-marking to the $X$-marking, and in each column connect the $X$-marking to the $O$-marking with an oriented line segment, such that the vertical segments always pass over the horizontal ones. This way we get closed curves: the diagram of the oriented link $\vec{L}$ given by the grid diagram. We call $\GG$ a \emph{grid diagram for} $\vec{L}$.

The converse is also true: Every oriented link in $S^3$ can be represented by a grid diagram. (See \cite[Chapter 3.]{ozsvath2015grid} for the easy proof.)

A \emph{cyclic permutation} of a planar grid diagram $\GG$ is a process in which we permute the rows or the columns of $\GG$ cyclically. Note that a cyclic permutation has no effect on the induced toroidal grid diagram, and that two different planar realizations of a toroidal grid diagram can always be connected by a sequence of cyclic permutations. Also $\vec{L}$, represented by the grid diagram stays the same after some cyclic permutations.

\subsection{Grid moves}

In 1995 Cromwell \cite{Cromwell199537} introduced a list of alterations of grid diagrams that are similar to Reidemeister moves for knot diagrams, and do not change the knot type.

First we consider the two main types of moves on planar grid diagrams, then we extend these to toroidal grid diagrams. For a detailed introduction of grid moves see \cite[Chapter 3.]{ozsvath2015grid}.

In a grid diagram $\GG$ every column determines a closed interval of real numbers that connects the height of its $O$-marking with the height of its $X$-marking. Consider a pair of consecutive columns in $\GG$, and suppose that the two intervals associated to them are either disjoint, or one is contained in the other (Figure \ref{fig:columncommutation}). We say that the diagram $\GG'$ differs from $\GG$ by a \emph{column commutation}, if it can be obtained by interchanging two consecutive columns of $\GG$ that satisfy the above condition. A \emph{row commutation} is defined analogously, using rows in place of columns. Column or row commutations collectively are called \emph{commutation moves}.


If we consider two consecutive columns in the grid diagram $\GG$, such that the interior of their corresponding intervals are not disjoint, but neither is contained by the other, and interchange these columns to get grid diagram $\GG'$, then $\GG$ and $\GG'$ are related by a \emph{cross-commutation}. We use the same name for the procedure interpreted on rows in place of columns. If $\vec{L}$ and $\vec{L}'$ are two links whose grid diagrams $\GG$ and $\GG'$ differ by a cross-commutation, then $\vec{L}$ and $\vec{L}'$ are related by a \emph{crossing change}.

While commutation moves preserve the grid number, the second type of grid moves will change it:

Let $\GG$ be an $n\times n$ grid diagram. We say that the $(n+1)\times(n+1)$ grid diagram $\GG'$ is the \emph{stabilization} of $\GG$ if it can be obtained from $\GG$ in the following way: Choose a marked square in $\GG$, and erase the marking in it, in the other marked square in its row and in the other marked square in its column. Then split the row and the column of the chosen marking in $\GG$ into two, that is, add a new horizontal and a new vertical line to get an $(n+1)\times(n+1)$ grid, see Figure \ref{fig:columncommutation}. Note that there are four ways to insert markings in the two new rows and columns to have a grid diagram, $\GG'$ can be any of these four.

The following lemma can be found as Lemma 3.2.2. in \cite{ozsvath2015grid}.

\begin{lemma}
Any stabilization can be expressed as a stabilization of the type shown in Figure \ref{fig:columncommutation} followed by a sequence of commutations.
\end{lemma}

\begin{figure}[h]
\centering
\includegraphics[width=0.8\textwidth]{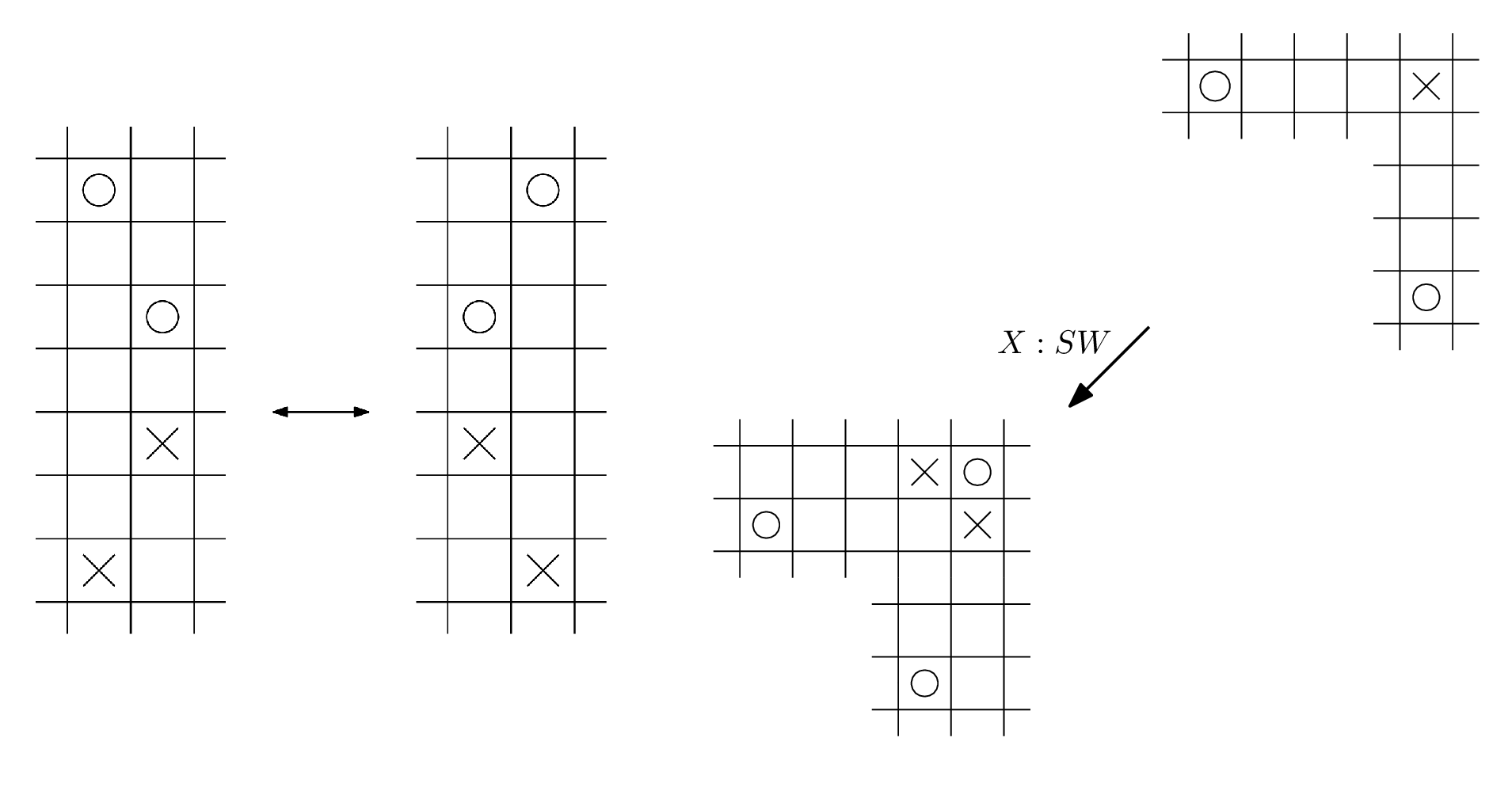}
\caption{A commutation (left) and a stabilization (right)}
\label{fig:columncommutation}
\end{figure}

The inverse of a stabilization is a \emph{destabilization}.
Commutation, stabilization and destabilization together are called \emph{grid moves} or Cromwell moves.

The following theorem of Cromwell \cite{Cromwell199537} shows us that grid moves and grid diagrams are really effective tools for creating knot invariants:
\begin{thm}
Two toroidal grid diagrams represent equivalent links if and only if there is a finite sequence of grid moves that transform one into the other.
\end{thm}

Grid moves can be transmitted naturally for the case of toroidal grid diagrams: two toroidal grid diagrams differ by a commutation move/a stabilization, if they have planar realizations that differ by a commutation move/a stabilization.

\subsection{Grid states and gradings}

We would like to define a version of grid homology, so first we need a chain complex. We recall here the concept of grid states, then, for the boundary map, we take rectangles connecting grid states. For a detailed introduction see \cite[Chapter 4.]{ozsvath2015grid}.

A \emph{grid state} $\gs x$ for a grid diagram $\GG$ with grid number $n$ consists of $n$ points in the torus such that each horizontal and each vertical circle contains exactly one element of $\gs x$. The set of grid states for $\GG$ is denoted by $\gs S(\GG)$.

Consider two grid states $\gs x, \gs y \in \gs S(\GG)$ that overlap in $n-2$ points in the torus. The four points left out are the corners of an embedded \emph{rectangle} $r$, that inherits an orientation from the torus. The boundary of $r$ is the union of oriented arcs, two of which are on the horizontal and two on the vertical circles. We say that the rectangle $r$ \emph{goes from} $\gs x$ \emph{to} $\gs y$ if the horizontal segments in $\partial r$ point from the components of $\gs x$ to the components of $\gs y$, while the vertical segments point from the components of $\gs y$ to the components of $\gs x$.

For $\gs x, \gs y \in \gs S(\GG)$, we use the notation $\Rect(\gs x, \gs y)$ for the set of rectangles going from $\gs x$ to $\gs y$.

We call a rectangle $r\in \Rect(\gs x, \gs y)$ an \emph{empty rectangle} if 
$$\gs x\cap \Int(r)=\gs y \cap \Int(r)=\emptyset.$$
The set of empty rectangles from $\gs x$ to $\gs y$ is denoted by $\rect(\gs x, \gs y)$.

We will need a generalization of the rectangle idea:
For $\gs x, \gs y\in \gs S(\GG)$ a (positive) \emph{domain} $\psi$ from $\gs x$ to $\gs y$ is a formal linear combination of the small squares of $\GG$, where the coefficients are non-negative integers, and the horizontal boundary segments of $\psi$ point from the components of $\gs x$ to the components of $\gs y$. The set of domains from $\gs x$ to $\gs y$ is denoted by $\pi(\gs x, \gs y)$.

To introduce grid homology, we equip grid states with two gradings, called the Maslov ($M_\OO$) and the Alexander grading ($A$). For the definitions and properties see \cite[Section 4.3.]{ozsvath2015grid}. Here we only state two important facts:

\begin{prop}
 Let $\GG$ be a toroidal grid diagram for a knot, and $\gs x, \gs y \in \gs S(\GG)$ two grid states that can be connected by some rectangle $r\in \Rect(\gs x, \gs y)$. Then,
\begin{equation}\label{Maslovrectangle}
M_\OO(\gs x)-M_\OO(\gs y)=1-2\cdot|r\cap \OO|+2\cdot |\gs x \cap \Int(r)|, \quad\mathrm{and}
\end{equation}
 \begin{equation}\label{Alexanderrectangle}
A(\gs x)-A(\gs y)=|r\cap \XX|-|r\cap \OO|.
\end{equation}
\end{prop}

\section{The $t$-modified grid complex}

First of all, we recall some basic notions that will play an important role afterwards.

A set $S\subset \RR$ is called \emph{well-ordered}, if every subset of $S$ has a \mbox{minimal element.}
We will use the notation 
$$\mathcal{R}'=\{\sum\limits_{\alpha \in A} U^\alpha | A\subset \mathbb{R}_{\geq 0}, \ A\ \mathrm{ is\ well-ordered}\ \} $$
for the \emph{long power series ring over $\mathbb{F}_2$}. (We denote by $\mathbb{F}_2$ the field with two elements.)

$\mathcal{R}'$ is indeed a ring with the following operations:
For $A$, $B\subset \mathbb{R}_{\geq 0}$ well-ordered sets let $a=\sum\limits_{\alpha\in A}U^\alpha$ and $b=\sum\limits_{\beta\in B}U^\beta$. Define the sum on $\mathcal{R}'$ in a way that $a+b=\sum\limits_{\gamma\in C}U^\gamma$, where $C=A\cup B\setminus A\cap B$ is the symmetric difference of the index sets.

Notice that $C$ is well-ordered, since the union, the intersection and the difference of well-ordered sets are also well-ordered.

The multiplication on $\mathcal{R}'$ is defined as the Cauchy-product of the elements, that is,
$$\left(\sum\limits_{\alpha\in A}U^\alpha\right)\cdot \left( \sum\limits_{\beta\in B}U^\beta\right)=\sum\limits_{\gamma\in A+B} \left(\sum\limits_{\substack{\alpha\in A, \beta\in B \\ \alpha+\beta=\gamma}}U^\gamma\right).$$
Here $A+B$ is the sumset of $A$ and $B$. This definition makes sense, because every $\gamma$ can be written as the sum of some $\alpha$ and $\beta$ in finitely many ways: if we suppose the opposite, then there exists an infinite, strictly decreasing sequence of either the $\alpha$'s or the $\beta$'s, which contradicts the fact that $A$ and $B$ are well-ordered sets.

Let $\cR_t\leq \cR'$ be the subring generated by the elements $U^t$ and $U^{2-t}$.

Now we are ready to define the chain complex which will give us the desired homologies.

\begin{defn}
For each $t\in [0,2]$ we define the \emph{$t$-modified grid complex} $tGC^-(\mathbb{G})$ as the free module over $\mathcal{R}$ generated by $\mathbf{S}(\mathbb{G})$, equipped with the $\mathcal{R}$-module endomorphism $\partial^-_t$, whose value on any $\gs x\in \mathbf{S}(\mathbb{G})$ is given by
\begin{equation}\label{eq:deltax}
\partial^-_t (\gs x) =\sum\limits_{\gs y\in \mathbf{S}(\mathbb{G})} \left(\sum\limits_{r\in \rect (\gs x,\gs y)} U^{t\cdot |\XX \cap r|+(2-t)\cdot |\OO\cap r|}\right) \cdot \gs y.
\end{equation}
\end{defn}


\begin{thm}\label{deltadelta0}
The operator $\partial^-_t :\ tGC^-(\mathbb{G})\rightarrow tGC^-(\mathbb{G})$ satisfies $\partial^-_t \circ \partial^-_t=0$.
\end{thm}

\begin{proof}[\textbf{\emph{Proof}}]
We follow the proof of Lemma 4.6.7 in \cite{ozsvath2015grid}.
Consider grid states $\gs x$ and $\gs z$. For a fixed $\psi \in \pi(\gs x,\gs z)$ denote by $N(\psi)$ the number of ways we can decompose $\psi$
as a juxtaposition of two empty rectangles $r_1 * r_2$. Notice that if $\psi=r_1*r_2$ for some
$r_1\in \rect(\gs x,\gs y)$ and $r_2\in \rect(\gs y,\gs z)$, then
$$|\XX \cap \psi|=|\XX\cap r_1|+|\XX\cap r_2| \quad \mathrm{and}\quad |\OO \cap \psi|=|\OO\cap r_1|+|\OO\cap r_2|.$$
It follows that the endomorphism $\partial^-_t \circ \partial^-_t$ can be expressed on $\gs x$ as
$$\partial^-_t \circ \partial^-_t(\gs x)=\sum\limits_{\gs z\in \gs S(\mathbb G) } \left(
\sum\limits_{\psi\in \pi(\gs x, \gs z)} N(\psi)\cdot U^{t\cdot |\XX \cap \psi|+(2-t)\cdot |\OO\cap \psi|}\right) \cdot \gs z. $$

Take a pair of empty rectangles $r_1\in \rect(\gs x,\gs y)$ and $r_2\in \rect(\gs y,\gs z)$ so that $N(\psi)>0$ holds for the domain $r_1 * r_2=\psi$. There are three basic cases (see Figure 4.4 in \cite{ozsvath2015grid} for an illustration):
\begin{itemize}
\item $\gs x \setminus (\gs x \cap \gs z)$ consists of 4 elements. Then, $r_1$ and $r_2$ do not share any common corners. In this case, $\psi$ can only be decomposed as $r_1*r_2$ or $r_2*r_1$. Therefore, $N(\psi)=2$.
\item $\gs x \setminus (\gs x \cap \gs z)$ consists of 3 elements. In this case, the local multiplicities of $\psi$ are all 0 or 1. The corresponding region in the torus is a hexagon with five corners of $90^\circ$, and one of $270^\circ$. Now $N(\psi)=2$ again, and the two decompositions can be obtained by cutting $\psi$ at the $270^\circ$ corner horizontally and vertically.
\item $\gs x=\gs z$. In this case, $r_1$ and $r_2$ intersect along two edges, and $\psi$ is an annulus. Since the rectangles are empty, the height or width of this annulus is 1 (called a thin annulus), and $N(\psi)=1$. Hence $|\XX\cap \psi|=|\OO\cap \psi|=1$. 
\end{itemize}

Contributions from the first two cases cancel in pairs, because we are working modulo 2. So it is enough to consider the terms where $\psi$ is a thin annulus. Note that every thin annulus is a proper choice for $\psi$.  Now we have
$$\partial^-_t \circ \partial^-_t(\gs x)= \left(
\sum\limits_{\psi \mathrm{\ is\ a\ thin\ annulus}} N(\psi)\cdot U^2
\right) \cdot \gs x = 2n\cdot U^2\cdot \gs x =0.$$
\end{proof}

Let us introduce a grading on the preferred basis of $tGC^-\mathbb{G}$.
\begin{defn}
For $\gs x\in \mathbf{S}(\mathbb{G})$ the \emph{$t$-grading} of $U^\alpha \cdot \gs x$ is defined as 
$$\gr(U^\alpha \cdot \gs x)=M(\gs x)-t\cdot A(\gs x)-\alpha,$$
where $M$ is the Maslov and $A$ is the Alexander function on grid states.
\end{defn}

\begin{prop}\label{delta_grading}
The differential $\partial^-_t$ drops the $t$-grading by one.
\end{prop}
\begin{proof}[\textbf{\emph{Proof}}]
Let $\gs x\in \gs S(\GG)$, and consider a non-zero term $U^{t\cdot |\XX \cap r|+(2-t)\cdot |\OO\cap r|}\cdot \gs y$ of the sum in the definition of $\partial^-_t(\gs x)$ in \eqref{eq:deltax}. Since $r\in \rect(\gs x, \gs y)$, from \eqref{Maslovrectangle} and \eqref{Alexanderrectangle} we know that 
$$M(\gs y)=M(\gs x) -1 + 2\cdot |\OO\cap r| \quad \mathrm{and}$$
$$ A(\gs y)=A(\gs x) - |\XX \cap r| + |\OO\cap r|.$$
Therefore,
$$ \gr(U^{t\cdot |\XX \cap r|+(2-t)\cdot |\OO\cap r|}\cdot \gs y)=M(\gs y)-t\cdot A(\gs y) - t\cdot |\XX \cap r| - (2-t)\cdot |\OO\cap r| =  $$
$$ = (M(\gs x) -1 + 2\cdot |\OO\cap r|)-t\cdot (A(\gs x) - |\XX \cap r| + |\OO\cap r|) - t\cdot |\XX \cap r| - (2-t)\cdot |\OO\cap r| = $$
$$ = M(\gs x) - t\cdot A(\gs x) -1 = \gr(\gs x) -1. $$

\end{proof}

Let $tGC^-_d(\mathbb{G})$ denote the vector space over $\mathbb{F}_2$ generated by those monomials $U^\alpha \cdot \gs x$, where $\gs x \in \gs S(\mathbb{G})$ and $\gr (U^\alpha \cdot \gs x)=d$. If $x \in tGC^-(\mathbb{G})$ lies in $tGC^-_d(\mathbb{G})$ for some $d$, we say that $x$ is \emph{homogeneous}.

We can define the homology of $(tGC^-(\mathbb{G}),\partial^-_t)$:

\begin{defn}
$$ tGH^-_d(\mathbb{G})= \faktor{\Ker(\partial_t^-)\cap tGC^-_d(\mathbb{G})}{\Im(\partial_t^-)\cap tGC^-_d(\mathbb{G})} $$
$$  tGH^-(\mathbb{G}) = \bigoplus_d tGH^-_d(\mathbb{G}).$$
\end{defn}

\section{Invariance}
\subsection{Commutation invariance}\label{commutation_invariance}

The aim of this section is to prove that $tGH^-$ is invariant under commutation. We introduce the same notations as Section 5.1 in \cite{ozsvath2015grid}.

\begin{thm}\label{thm:commutationinvariance}
Let $\GG$ and $\GG'$ be two grid diagrams that differ by a commutation move. Then 
$$tGH^-(\mathbb{G})\cong tGH^-(\mathbb{G'})$$ as graded $\cR_t$-modules.
\end{thm}
\begin{proof}
Consider two toroidal grid diagrams $\GG$ and $\GG'$ with grid number $n$, such that this latter can be obtained from $\GG$ by a commutation move. Draw these two diagrams onto the same torus in a way that the $X$- and $O$-markings are fixed, and two of the vertical circles are curved according to the right diagram of Figure \ref{fig:commutation_sametorus}. Denote the horizontal circles of $\GG$ by \mbox{$\mathbb{\alpha}=\{\alpha_1,...\alpha_n\}$} and its vertical circles by $\mathbb{\beta}=\{\beta_1,... \beta_n\}$. Then the set of horizontal circles of $\GG'$ is also $\mathbb{\alpha}$, but its vertical circles are different: $\mathbb{\gamma}=\{\beta_1,...,\beta_{i-1}, \gamma_i, \beta_{i+1},...,\beta_n\}$. Here we use the labeling compatible with the cyclic ordering of the toroidal grid, namely, $\beta_{k+1}$ for $k=1,...,n-1$ is the vertical circle immediately to the east of $\beta_k$. 

\begin{figure}[h]
\centering
\includegraphics[width=0.7\textwidth]{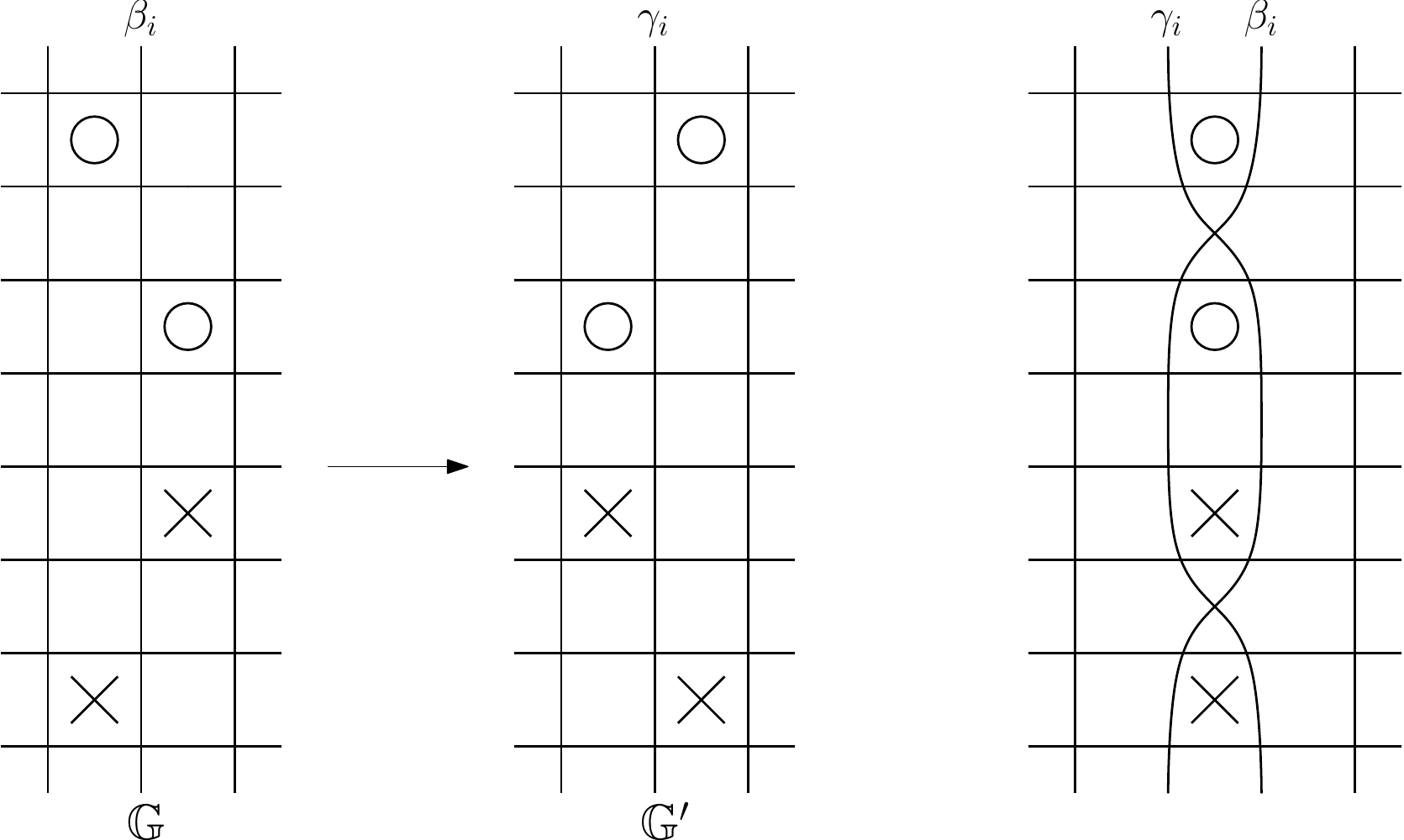}
\caption{Draving $\GG$ and $\GG'$ on the same torus. The place of the $X$- and $O$-markings is fixed, but the two vertical circles, $\beta_i$ and $\gamma_i$ are curved.}
\label{fig:commutation_sametorus}
\end{figure}

Draw these vertical circles so that $\beta_i$ meet $\gamma_i$ perpendicularly in two points, that do not lie on any of the horizontal circles. Denote these points by $a$ and $b$ due to the following convention: Take the complement of $\beta_i \cup \gamma_i$ in the grid. This consists of two bigons intersecting in $a$ and $b$. Consider the one, of which the western boundary is a part of $\beta_i$, and the eastern boundary is a part of $\gamma_i$. Let $a$ be the southern and $b$ the northern vertex of this bigon.

To make a connection between the chain complex of $\GG$ and the chain complex of $\GG'$, suitable pentagons, defined below will do a good service.

\begin{defn}
Consider two grid states $\gs x\in \gs S(\GG)$ and $\gs y'\in \gs S(\GG')$. A \emph{pentagon $p$ from $\gs x$ to $\gs y'$} is an embedded disk in the torus that satisfies the following conditions:
\begin{itemize}
\item The boundary of $p$ is the union of five arcs lying in $\alpha_j$, $\beta_j$ or $\gamma_i$ for $i$ and for some $j$.
\item Four corners of $p$ are in $\gs x \cup \gs y'$.
\item If we consider any corner point of $p$, it is the intersection of two curves of $\{\alpha_j, \beta_j, \gamma_i\}_{j=1}^n$. These two curves divide a small disk on the torus into four quadrants, and $p$ intersects exactly one of them.
\item The horizontal segments in $\partial p$ point from the components of $\gs x$ to the components of $\gs y'$.
\end{itemize}
\end{defn}

We use the notation $\Pent(\gs x, \gs y')$ for the set of pentagons going from $\gs x$ to $\gs y'$ \mbox{(Figure \ref{fig:pentagon}).}

\begin{figure}[h]
\centering
\includegraphics[width=0.5\textwidth]{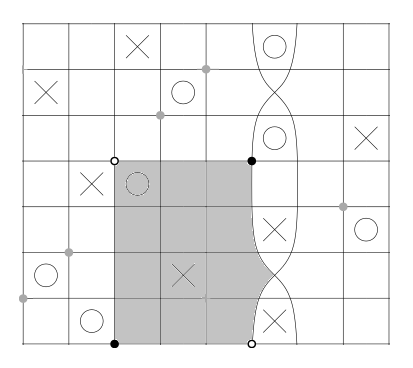}
\caption{An empty pentagon}
\label{fig:pentagon}
\end{figure}

Observe that the set $\Pent(\gs x, \gs y')$ consists of at most one element, and it is empty unless $\gs x$ and $\gs y'$ share exactly $n-2$ elements. From the above properties of a pentagon, it follows that its fifth corner point is the distinguished point $a$.

Pentagons from $\gs y'$ to $\gs x$ are defined similarly. The only difference is in the forth condition: the horizontal boundary segments point from the components of $\gs y'$ to the components of $\gs x$. The fifth vertex of such pentagons is $b$.

We call a pentagon $p\in \Pent(\gs x, \gs y')$ an \emph{empty pentagon} if 
$$\gs x \cap \Int(p)=\gs y' \cap \Int(p)=\emptyset.$$
The set of empty pentagons from $\gs x$ to $\gs y'$ is denoted by $\pent(\gs x, \gs y')$. (See Figure \ref{fig:pentagon}.)

Define the $\mathcal{R}$-module map $P: tGC^-(\mathbb{G}) \rightarrow tGC^-(\mathbb{G'})$ by the formula:
$$P(\gs x)= \sum\limits_{\gs {y'}\in \mathbf{S}(\mathbb{G'})}\left(\sum\limits_{p\in \pent (\gs x,\gs {y'})} U^{t\cdot |\XX \cap p|+(2-t)\cdot |\OO\cap p|}\right)\cdot \gs {y'} \quad \mathrm{for\ each\ }\gs x\in \gs S(\GG).$$

\begin{lemma}
For any $\gs x\in \mathbf{S(\mathbb{G)}}$, the map $P$ preserves the $t$-grading, i.e. \mbox{$\gr (P(\gs x))=\gr (\gs x)$.}
\end{lemma}
\begin{proof}[\textbf{\emph{Proof}}]
Consider one term of the sum. Let $\gs {y'}\in \gs S(\GG')$
such that there exists a pentagon $p\in \pent (\gs x,\gs {y'})$. It is easy to verify the following formulae, for  the details see Lemma 5.1.3. in \cite{ozsvath2015grid}:
$$M(\gs x)-M(\gs {y'})=-2\cdot |p\cap \OO|+2\cdot |\gs x \cap \Int (p)|,$$
$$A(\gs x)-A(\gs {y'})=|p\cap \XX| -|p \cap \OO|.$$
Using these we conclude that
$$\gr (U^{t\cdot |\XX \cap p|+(2-t)\cdot |\OO\cap p|}\cdot \gs {y'})=M(\gs {y'})- t\cdot A(\gs {y'})-(t\cdot |\XX \cap p|+(2-t)\cdot |\OO \cap p|)=$$
$$= M(\gs x)+2\cdot |\OO \cap p|-t\cdot (A(\gs x) -|\XX \cap p|+|\OO \cap p|)-(t\cdot |\XX\cap p|+(2-t)\cdot |\OO\cap p|)=$$
$$= M(\gs x)-t\cdot A(\gs x) = \gr (\gs x).$$
Hence, $\gr (P(\gs x))=\gr (\gs x)$.
\end{proof}

\begin{prop}\label{Pchainmap}
The map $P$ is a chain map.
\end{prop}
\begin{proof}[\textbf{\emph{Proof}}]
We need to show that $(P \circ \partial_t^-)(\gs x)=(\partial_t^- \circ P)(\gs x)$ for every $\gs x \in \SG$, which means the same as  $(P \circ \partial_t^-+\partial_t^- \circ P)(\gs x)=0$.

First we generalize the notion of a domain. By cutting the torus along the circles of $\alpha$, of $\beta$ and $\gamma_i$ we obtain a union of triangles, rectangles and pentagons, call them elementary regions. Fix grid states $\gs x\in \gs S(\GG)$ and $\gs y'\in \gs S(\GG')$. Let a domain $\psi$ from $\gs x$ to $\gs y'$ be a formal sum of the closures of the elementary regions with non-negative integer coefficients, which satisfies the condition that the horizontal boundary segments of $\psi$ point from the components of $\gs x$ to the components of $\gs y'$.

As in the proof of Theorem \ref{deltadelta0}, we can easily see that 
$$(P \circ \partial_t^- + \partial_t^- \circ P)(\gs x)= \sum \limits_{\gs {z'} \in \mathbf{S}(\mathbb{G'})} \left(\sum \limits_{\psi\in \pi (\gs x, \gs {z'})} N(\psi)\cdot U^{t\cdot |\XX \cap \psi|+(2-t)\cdot |\OO\cap \psi|}\right)\cdot \gs {z'},$$
where $N(\psi)$ denotes the number of decompositions of $\psi$ as a juxtaposition of an empty rectangle followed by an empty pentagon or as an empty pentagon followed by an empty rectangle. 

Consider a domain $\psi$ such that $N(\psi)>0$. We examine three basic cases (for illustration see Figure 5.5 and 5.6 in \cite{ozsvath2015grid}):
\begin{itemize}
\item $\gs x\setminus (\gs x \cap \gs z')$ consists of 4 elements. Then $\psi$ can be obtained as the juxtaposition of a pentagon $p$ and a rectangle $r$ which do not share any common corner. In this case the only decompositions of $\psi$ are $p*r$ and $r*p$. Therefore, $N(\psi)=2$.
\item $\gs x\setminus (\gs x \cap \gs z')$ consists of 3 elements. In this case either all of the local multiplicities of $\psi$ are 0 and 1, or also 2 appears as a local multiplicity. Either way, $\psi$ has seven corners, one of which is of degree $270^\circ$. We get two different decompositions of $\psi$ for a pentagon and a rectangle by cutting it at this corner in two different directions, thus $N(\psi)=2$.
\item $\gs x\setminus (\gs x \cap \gs z')$ consists of 1 element. Then $\psi$ goes around the torus. This can happen in two ways: the juxtaposition is either a vertical thin annulus together with a triangle or a horizontal thin annulus minus a triangle. In the former case, $\psi$ decomposes uniquely into a thin pentagon and a thin rectangle. In the latter case, there are two ways to cut $\psi$ near $a$: along $\beta_i$ or along $\gamma_i$. Therefore, $N(\psi)=2$.
\end{itemize}

Since we work over $\mathbb{F}_2$, it is enough to consider the cases where $N(\psi)=1$ (then $\psi$ is a thin annulus with a triangle). It can be shown that the contributions coming from these domains cancel in pairs. For the details see Lemma 5.1.4. in \cite{ozsvath2015grid}.
\end{proof}

Now we define an analogous $\cR_t$-module homomorphism $P': tGC^-(\mathbb{G'}) \rightarrow tGC^-(\mathbb{G)}$. For a grid state $\gs x \in \gs S(\GG')$, let
$$P'(\gs {x'})=\sum \limits_{\gs {y}\in \mathbf{S}(\mathbb{G})}\left(\sum\limits_{p\in \pent (\gs {x'},\gs y)} U^{t\cdot |\XX \cap p|+(2-t)\cdot |\OO\cap p|}\right)\cdot \gs y. $$
Henceforth we will show that the two maps $P$ and $P'$ are homotopy inverses of each other. For this aim we need the following notion:

\begin{defn}
Consider two grid states $\gs x$ and $\gs y\in \gs S(\GG)$. A \emph{hexagon} $h$ from $\gs x$ to $\gs y$  is an embedded disk in the torus that satisfies the following conditions:
\begin{itemize}
\item The boundary of $h$ is the union of six arcs lying in $\alpha_j$, $\beta_j$ or $\gamma_i$ for $i$ and for some $j$.
\item Four corners of $h$ are in $\gs x \cup \gs y$, and the two further corners are $a$ and $b$.
\item Any corner point of $h$ is the intersection of two curves of $\{\alpha_j, \beta_j, \gamma_i\}_{j=1}^n$. These two curves divide the torus into four quadrants, and $h$ lies in exactly one of them.
\item The horizontal segments of $\partial h$ point from the components of $\gs x$ to the components of $\gs y$.
\end{itemize}
We use the notation $\Hex(\gs x, \gs y)$ for the set of hexagons going from $\gs x$ to $\gs y$. 
\end{defn}

We call a hexagon $h\in \Hex(\gs x, \gs y)$ an \emph{empty hexagon} if
$$\gs x\cap \Int (h)=\gs y \cap \Int(h)=\emptyset.$$
The set of empty hexagons from $\gs x$ to $\gs y$ is denoted by $\hex(\gs x, \gs y)$.

Define the $\mathcal{R}$-module homomorphism $H: tGC^-(\GG)\rightarrow tGC^-(\GG)$ for each $\gs x \in \SG$ by the formula:
$$H(\gs x)=\sum \limits_{\gs {y}\in \gs S(\GG)}\left(\sum\limits_{h\in \hex (\gs x,\gs y)} U^{t\cdot |\XX \cap h|+(2-t)\cdot |\OO\cap h|}\right)\cdot \gs y.$$
An analogous map $H':tGC^-(\GG')\rightarrow tGC^-(\GG')$ can be defined by counting empty hexagons from $tGC^-(\GG')$ to itself.

\begin{lemma}\label{H_chain_homotopy}
The map $H: tGC^-(\mathbb{G})\rightarrow tGC^-(\mathbb{G})$ provides a chain homotopy from the chain map $P' \circ P$ to the identity map on $tGC^-(\mathbb{G}).$
\end{lemma}
\begin{proof}[\textbf{\emph{Proof}}]

First we show that $H$ increases the $t$-grading by 1. Consider a non-zero term of the sum in the definition of $H(\gs x)$. Let $\gs y \in \gs S(\GG)$ such that there exists a hexagon $h\in \hex(\gs x, \gs y)$. This $h$ can be augmented to a rectangle $r\in \Rect(\gs x, \gs y)$ by adding a bigon with vertices $a$ and $b$, that contains exactly one $X$-marking and one $O$-marking. Applying Equations \eqref{Maslovrectangle} and \eqref{Alexanderrectangle} we can derive that 
$$M(U^{t\cdot |\XX\cap h|+(2-t)\cdot |\OO\cap h|}\cdot \gs y)=M(\gs x)+1 \quad \mathrm{and} \quad A(U^{t\cdot |\XX\cap h|+(2-t)\cdot |\OO\cap h|}\cdot \gs y)=A(\gs x).$$

To have that $H$ is a chain homotopy from $P' \circ P$ to the identity map on $tGC^-(\mathbb{G})$, we have to verify that
$$\partial_t^- \circ H + H \circ \partial_t^-=\mathrm{Id}-P'\circ P, \quad \mathrm{that\ is,}$$
$$(\partial_t^- \circ H + H \circ \partial_t^-+P'\circ P)(\gs x)=\gs x \quad \mathrm{for\ any\ }\gs x\in \gs S(\GG).$$

The idea of the proof is the same as in the proof of Theorem \ref{deltadelta0} and Proposition \ref{Pchainmap}. For a domain $\psi\in \pi(\gs x, \gs z)$ denote by $N(\psi)$ the number of ways $\psi$ can be decomposed as one of the followings:
\begin{itemize}
\item $\psi=r*h$, where $r$ is an empty rectangle and $h$ is an empty hexagon;
\item $\psi=h*r$, where $h$ is an empty hexagon and $r$ is an empty rectangle;
\item $\psi=p*p'$, where $p$ is an empty pentagon from $\gs S(\GG)$ to $\gs S(\GG')$ and $p'$ is an empty pentagon from $\gs S(\GG')$ to $\gs S(\GG)$.
\end{itemize}

Obviously,
\begin{equation}\label{eq:Hchainmap}
\partial_t^- \circ H + H \circ \partial_t^-+P'\circ P)(\gs x)=\sum \limits_{\gs z \in \gs S(\GG)}\left(\sum\limits_{\psi \in \pi(\gs x, \gs z)}N(\psi)\cdot U^{t\cdot |\XX \cap \psi|+(2-t)\cdot |\OO\cap \psi|}\right)\cdot \gs z.
\end{equation}

For a domain $\psi$ for which $N(\psi)>0$, we have three basic cases again (for illustration see Figure 5.8 and 5.9 in \cite{ozsvath2015grid}):
\begin{itemize}
\item $\gs x \setminus (\gs x\cap \gs z)$ consists of 4 elements. Then $\psi$ can be decomposed into the juxtaposition of a hexagon $h$ and a rectangle $r$ which do not share any common corner. In this case the only decompositions of $\psi$ are $h*r$ and $r*h$. Therefore, $N(\psi)=2$.
\item $\gs x \setminus (\gs x\cap \gs z)$ consists of 3 elements. Then $\psi$ has eight vertices, and there are two possibilities: either seven corners are $90^\circ$ and one is $270^\circ$, or five corners are $90^\circ$ and three are $270^\circ$ (then two of the $270^\circ$ corners are at $a$ and $b$). In the former case, we get two different decompositions of $\psi$ by cutting at the $270^\circ$ corner in two different directions. In the latter case $N(\psi)=2$ again, since at one of the three corners we can cut in two different directions, while at the other two corners the direction for cutting is uniquely determined.
\item $\gs x= \gs z$. Then $\psi$ is an annulus enclosed by $\beta_i$ and one of the consecutive vertical circles. In this case $N(\psi)=1$.
\end{itemize}
There is a unique domain $\psi$ which has non-zero contribution in Equation \eqref{eq:Hchainmap}.

For this, $N(\psi)=1$ and $\psi\in \pi(\gs x, \gs x)$ contains no marking. Therefore the only term of the sum is $\gs x$.
\end{proof}

Gathering the above propositions, we finished the proof of Theorem \ref{thm:commutationinvariance}.
\end{proof}

\subsection{Stabilization invariance}\label{stabilizationinv}

Our goal in this section is to check whether $tGH^-$ is invariant under stabilization. In fact, we will need some modifications to make it invariant. For this, we will use the following notions and notations:

Let $X=\bigoplus\limits_d X_d$ and $Y=\bigoplus\limits_d Y_d$ be graded vector spaces, such that in at least one of them, the set of grades appearing in it is bounded above. Their tensor product $X\otimes Y=\bigoplus\limits_d (X\otimes Y)_d$ is the graded vector space with 
$$(X\otimes Y)_d=\bigoplus\limits_{d_1+d_2=d} X_{d_1}\otimes Y_{d_2}.$$

For a graded vector space $X$ and $s\in \RR$ the \emph{shift of $X$ by $s$} denoted by $X[\![s]\!]$, is the graded vector space that is isomorphic to $X$ as a vector space and the grading on $X[\![s]\!]$ is given by the relation $X[\![s]\!]_d=X_{d+s}$.

Consider the two-dimensional graded vector space $W$ with one generator of grading $1-t$ and another of grading $0$. Take any other graded vector space $X$. Then
\begin{equation}
X\otimes W \cong X\shift{1-t}\oplus X.
\end{equation}  

This procedure can be iterated, for example:
$$(X\otimes W)\otimes W=X\otimes W^{\otimes 2}\cong X\shift{2-2t}\oplus X\shift{1-t}\oplus X\shift{1-t}\oplus X$$

Now we can introduce the true invariant, which is an equivalence class of pairs consisting of a group and an integer. Concretely, for a grid diagram $\GG$, consider the pair $(tGH^-(\GG), n)$, where $n$ is the grid number of $\GG$. Let $\GG'$ be another grid diagram with grid number $n'$, such that $n\leq n'$. The two pairs $(tGH^-(\GG), n)$ and $(tGH^-(\GG'), n')$ are called equivalent if 
$$tGH^-(\GG')\cong tGH^-(\GG)\otimes W^{\otimes(n'-n)}.$$

Henceforward in this section we will prove that the equivalence class $\left[(tGH^-(\GG), n)\right]$ is indeed invariant under stabilization.

\begin{rem}
$\left[(tGH^-(\GG), n)\right]$ is invariant under commutation, since we have seen in the previous section that commutation moves do not change $tGH^-(\GG)$, and the grid number $n$ obviously remains the same.
\end{rem}

Let $\GG$ be a grid diagram. By performing a stabilization of type $X:SW$, we get the diagram $\GG'$. Number the markings in the way that $O_1$ is the newly-introduced $O$-marking, $O_2$ is in the consecutive row below $O_1$, $X_1$ and $X_2$ lie in the same row as $O_1$ and $O_2$, respectively, i.e. 
$\begin{array}{c|c}
X_1 & O_1\\
\hline
&X_2
\end{array}$.

Denote by $c$ the intersection point of the new horizontal and vertical circles in $\GG'$. Considering this point, we can partition the grid states of the stabilized diagram $\GG'$ into two parts, depending on whether or not they contain the intersection point $c$. Define the sets $\gs I(\GG')$ and $\gs N(\GG')$ so that $\gs x\in \gs I(\GG')$ if $c$ is included in $\gs x$, and $\gs x\in \gs N(\GG')$ if $c$ is not included in $\gs x$. Now $\gs S(\GG')=\gs I(\GG')\cup \gs N(\GG')$ gives a decomposition of $tGC^-(\GG')\cong I\oplus N$, where $I$ and $N$ denote the $\cR_t$-modules spanned by the grid states of $\gs I(\GG')$ and $\gs N(\GG')$ respectively.

There is a one-to-one correspondence between grid states of $\gs I(\GG')$ and grid states of $\gs S(\GG)$: Let 
$$e:\gs I(\GG')\rightarrow \gs S(\GG), \qquad \gs x\cup \{c\} \mapsto \gs x.$$

From this point on, we will work with certain types of domains again:

\begin{defn}
For grid states $\gs x\in \gs S(\GG')$ and $\gs y\in \gs I(\GG')$ a positive domain \mbox{$p\in \pi(\gs x, \gs y)$} is called \emph{into $L$} or of type $iL$, if it is trivial, or if it satisfies the following conditions:
\begin{itemize}
\item At each corner point in $\gs x \cup \gs y \setminus \{c\}$ at least three of the four adjoining squares have local multiplicity 0.
\item Three of the four squares attaching at the corner point $c$ have the same local multiplicity $k$, and at the southwest square meeting $c$ the local multiplicity of $p$ is $k-1$.
\item The grid state $\gs y$ has $2k+1$ components that are not in $\gs x$.
\end{itemize}
The set of domains of type $iL$ is denoted by $\pi^{iL}(\gs x, \gs y)$ (see Figure \ref{fig:destabilizationdomains}).
\end{defn}

\begin{defn}
Analogously, for grid states $\gs x\in \gs S(\GG')$ and $\gs y\in \gs I(\GG')$ a positive domain $p\in \pi(\gs x, \gs y)$ is called \emph{into $R$} or of type $iR$, if it satisfies the following conditions:
\begin{itemize}
\item At each corner point in $\gs x \cup \gs y \setminus \{c\}$ at least three of the four adjoining squares have local multiplicity 0.
\item Three of the four squares attaching at the corner point $c$ have the same local multiplicity $k$, and at the southeast square meeting $c$ the local multiplicity of $p$ is $k+1$.
\item The grid state $\gs y$ has $2k+2$ components that are not in $\gs x$.
\end{itemize}
The set of domains of type $iR$ is denoted by $\pi^{iR}(\gs x, \gs y)$ (see Figure \ref{fig:destabilizationdomains}).
\end{defn}

\begin{figure}[h]
\centering
\includegraphics[width=0.9\textwidth]{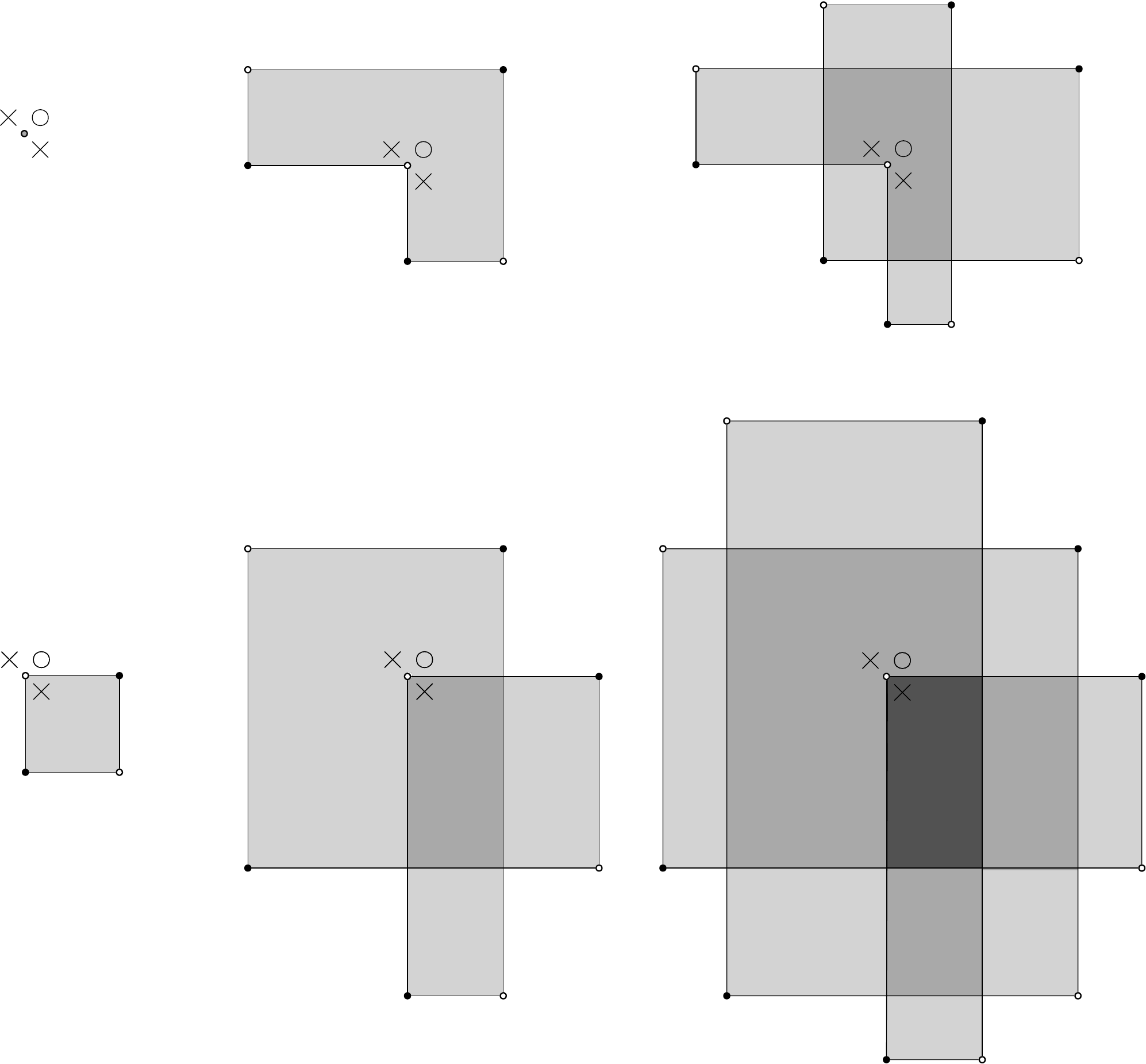}
\caption{Destabilization domains of type $iL$ (first row) and $iR$ (second row)}
\label{fig:destabilizationdomains}
\end{figure}

The domains of type $iL$ and $iR$ are called \emph{destabilization domains}. We use the notation $\pi^D=\pi^{iL}\cup \pi^{iR}$.

\begin{defn}
Let the \emph{complexity} of the trivial domain be one. For other destabilization domains the complexity tells the number of horizontal segments in the boundary of the domain.
\end{defn}

For example the upper right domain on Figure \ref{fig:destabilizationdomains} has complexity 5.
Observe that the domains of type $iL$ are the destabilization domains with odd complexity, while the domains of type $iR$ are the ones with even complexity.

\begin{lemma}\label{destabtorectangles}
Let $\gs x, \gs y\in \gs S(\GG')$ be grid states such that there exists a destabilization domain $p\in \pi(\gs x, \gs y)$ of complexity $k$. Then there is a sequence of grid states $\gs x_1,...,\gs x_k$ such that $\gs x_1=\gs x$, $\gs x_k=\gs y$, and empty rectangles $r_1,...,r_{k-1}$ with $r_i\in \rect(\gs x_i, \gs x_{i+1})$, so that $p$ decomposes as $r_1*...*r_{k-1}$. Among such rectangle sequences, there is a unique one in which every rectangle $r_i$ has an edge on the distinguished vertical circle going through $c$.
\end{lemma}

The proof of this lemma can be found as the proof of Lemma 13.3.11. in \cite{ozsvath2015grid}.

Now we are ready to introduce the chain map which gives the isomorphism between $tGC^-(\GG')$ and $tGC^-(\GG)\shift{1-t}\oplus tGC^-(\GG)$. For this, we will use the abbreviations $\Xbar:=\{X_1, X_3, ... , X_n \}$ and $\Obar:=\{O_1, O_3, ... , O_n\}$.

\begin{defn}
${}$\\
The $\cR_t$-module homomorphisms $D^{iL}:\ tGC^-(\GG') \rightarrow tGC^-(\GG){[\![1-t]\!]}$ and $D^{iR}:\ tGC^-(\GG') \rightarrow tGC^-(\GG)$ are defined on a grid state $\gs x\in \gs S(\GG)$ in the following way:
$$D^{iL}(\gs x):=\sum \limits_{\gs y\in I(\GG')}\sum\limits_{p\in \pi^{iL}(\gs x, \gs y)}U^{t\cdot |\Xbar \cap p|+(2-t)\cdot|\Obar\cap p|}\cdot e(\gs y)$$
$$D^{iR}(\gs x):=\sum \limits_{\gs y\in I(\GG')}\sum\limits_{p\in \pi^{iR}(\gs x, \gs y)}U^{t\cdot |\Xbar \cap p|+(2-t)\cdot|\Obar\cap p|}\cdot e(\gs y)$$
\end{defn}

\begin{prop}
$D^{iL}$ increases the grading by $(1-t)$, whereas $D^{iR}$ preserves the grading.
\end{prop}
\begin{proof}
Let $\gs x, \gs y \in \mathbf{S}(\mathbb{G'})$. First observe that $\gr(e(\gs y))=\gr(\gs y)+1-t$. To get the result about $D^{iL}$, our goal is to show that for $p\in \pi^{iL}(\gs x,\gs y)$
$$\gr(\gs x)=\gr (U^{t \cdot |\Xbar \cap p|+(2-t)\cdot |\Obar\cap p|}\cdot \gs y).$$
Suppose that the complexity of $p$ is $k\in \mathbb{N}$. Then by Lemma \ref{destabtorectangles}, $p$ can be decomposed as the juxtaposition of $k-1$ rectangles. Consider the Maslov and the Alexander gradings. From Equations \eqref{Maslovrectangle} and \eqref{Alexanderrectangle} we have that 
$$M(\gs x)-M(\gs y)=(1-2\cdot|\OO\cap r_1|)+(1-2\cdot|\OO\cap r_2|)+...+(1-2\cdot|\OO\cap r_{k-1}|)= $$
$$=k-1 -2\cdot |\OO\cap p|=-2\cdot|\Obar\cap p|,$$
$$A(\gs x)-A(\gs y)=|\XX\cap p|-|\OO\cap p|=|\Xbar\cap p|-|\Obar\cap p|$$
because an $iL$-type domain contains $O_1$ and $X_2$ with the same multiplicity.
$$\gr(\gs x)-\gr (\gs y)=(M(\gs x)-M(\gs y))-t\cdot (A(\gs x)-A(\gs y))=-t\cdot|\Xbar\cap p|-(2-t)\cdot |\Obar\cap p|,$$ from which we get
$$\gr(\gs x)=\gr(U^{t \cdot |\Xbar \cap p|+(2-t)\cdot|\Obar \cap p|}\cdot e(\gs y))-1+t.$$

Now let $p\in \pi^{iR}(\gs x, \gs y)$ with complexity $k$.
$$M(\gs x)-M(\gs y)=k-1-2\cdot |\OO\cap p|=1-2\cdot |\Obar \cap p|,$$
$$A(\gs x)-A(\gs y)=|\XX\cap p|-|\OO\cap p|=|\Xbar\cap p|-|\Obar\cap p|+1$$ 
because in an $iR$-type domain the multiplicity of $X_2$ is one greater than the multiplicity of $O_1$.
$$\gr(\gs x)-\gr (\gs y)=(M(\gs x)-M(\gs y))-t\cdot (A(\gs x)-A(\gs y))=1-t-t \cdot |\Xbar \cap p|-(2-t)\cdot |\Obar \cap p|,$$ from which we get
$$\gr(\gs x)=\gr(U^{t \cdot |\Xbar \cap p|+(2-t)\cdot|\Obar \cap p|}\cdot e(\gs y)).$$

\end{proof}

Let $C:=tGC^-(\GG){[\![1-t]\!]}\oplus tGC^-(\GG)$, and consider the map $\partial: C\rightarrow C$ such that $\partial(x,y)=(\partial_t^-(x), \partial_t^-(y))$ holds for any $(x,y)\in C$. Obviously, the pair $(C, \partial)$ is a chain complex.

\begin{defn}
${}$\\
Let $D: tGC^-(\GG') \rightarrow C$ be the destabilization map defined for any $ x \in tGC^-(\GG')$ as
$$D( x):=(D^{iL}( x), D^{iR}(y))\in C.$$
\end{defn}

The map $D^{iL}$ can be decomposed as $D^{iL}=D^{iL}_1+D^{iL}_{>1}$, where  the subscript indicates the restriction on the complexity of the destabilization domains. Using this, we can draw the following schematic picture for $D$, where the top row represents $tGC^-(\GG')$ with its decomposition as $I\oplus N$, and the bottom row shows $C$. The map $D$ can be seen in the arrows connecting the two rows.

$$\begin{diagram}
I& \rCorresponds & N \\
\dTo^{D^{iL}_1} & \ldTo^{D^{iL}_{>1}} & \dTo^{D^{iR}}\\
tGC^-(\GG)_{[\![1-t]\!]} & \oplus & tGC^-(\GG)
\end{diagram}$$

\begin{prop}
The destabilization map $D$ is a chain map.
\end{prop}
The proof of this proposition is based on counting regions on the grid diagram, and it is analogous to the proof of Lemma 13.3.13. in \cite{ozsvath2015grid}.

\begin{prop}\label{quasionfactors}
Suppose that $C$ and $\widetilde{C}$ are graded chain complexes over $\cR_t$, such that they are free modules, and the grades appearing in them is bounded above. Let $\alpha\in \RR_{\geq 0}$ and $f: C\rightarrow \widetilde{C}$ be a graded chain map. Then $f$ is a quasi-isomorphism if and only if it induces a quasi-isomorphism $\bar{f}: \faktor{C}{U^\alpha \cdot C}\rightarrow \faktor{\widetilde{C}}{U^\alpha \cdot \widetilde{C}}.$
\end{prop}
\begin{proof}
Let $(C, \partial)$ and $(C',\partial')$ be two chain complexes over $\cR_t$. The \emph{mapping cone} of a chain map $f:C\rightarrow C'$ is the chain complex $\Cone(f)=(C\oplus C', \partial_{\Cone})$, where the differential $\partial_{\Cone}$ for an element $(c,c')\in C\oplus C'$ is defined as
$$\partial_{\Cone}(c,c')=(-\partial(c),\partial(c')+f(c)).$$
Observe that $\faktor{\Cone(f)}{U^\alpha\cdot \Cone(f)}\cong \Cone(\bar{f})$.

\begin{lemma}\label{quasiisomorphism}
For $C$, $C'$ chain complexes, a chain map $f:C\rightarrow C'$ is a quasi-isomorphism if and only if $H(\Cone(f))=0$.
\end{lemma}
For the proof of this lemma see \cite[Appendix A, Corollary A.3.3.]{ozsvath2015grid}.

\begin{lemma}\label{HCnot0}
Suppose that $C$ is a free, graded chain complex over $\cR_t$ that is bounded above. Then $H(C)\neq 0$ if and only if $H\left(\faktor{C}{U^\alpha\cdot C}\right)\neq 0$ for a fixed $\alpha\in \RR_{>0}$.
\end{lemma}
\begin{proof}[Proof of Lemma \ref{HCnot0}]
We assumed that $C$ is free, thus there exists a short exact sequence
$$\begin{diagram}
0 &\rTo& C &\rTo^{\cdot U^\alpha}& C &\rTo& \faktor{C}{U^\alpha\cdot C} &\rTo& 0
\end{diagram}$$
Considering the associated long exact sequence, it is easy to see that if $H(C)=0$, then $H\left(\faktor{C}{U^\alpha\cdot C}\right)=0$.

Now suppose that $H(C)\neq 0$. Since $C$ is bounded above, $H(C)$ has a homogeneous, non-zero element $x$ with maximal grading. Then, $x$ cannot be of the form $y\cdot U^\alpha$ for any $y\in H(C)$. Therefore $x$ must inject to $H\left(\faktor{C}{U^\alpha\cdot C}\right)$, and this way we got a non-zero element of $H\left(\faktor{C}{U^\alpha\cdot C}\right)$.
\end{proof}

From Lemma \ref{quasiisomorphism} we know that the map $f$ is a quasi-isomorphism if and only if $H(\Cone(f))=0$. According to Lemma \ref{HCnot0}, this is equivalent to \mbox{$H\left(\faktor{\Cone(f)}{U^\alpha\cdot \Cone(f)}\right)\neq 0$.} By our observation, this holds if and only if $H(\Cone(\bar{f}))=0$, which, by Lemma \ref{quasiisomorphism} again, is equivalent to $\bar{f}$ being a quasi-isomorphism.
\end{proof}

Now let us use the following notations.
$$C_1:=tGC^-(\GG') \quad \mathrm{and} \quad D_1:=tGC^-(\GG)\oplus tGC^-(\GG)\shift{1-t},$$
$$C_2:=\faktor{C_1}{U^t \cdot C_1} \quad \mathrm{and} \quad D_2:=\faktor{D_1}{U^t \cdot D_1},$$
$$C_3:=\faktor{C_2}{U^{2-t}\cdot C_2} \quad \mathrm{and} \quad D_3:=\faktor{D_2}{U^{2-t}\cdot D_2}.$$ Furthermore, define $f=D: C_1\rightarrow D_1$ and let $\bar{f}: C_2 \rightarrow D_2$ be the map induced by $f$. Finally, define $\bar{\bar{f}}: C_3 \rightarrow D_3$, the map induced by $\bar{f}$.
Apply Proposition \ref{quasionfactors} for $\alpha=2-t$. Then we get that if $\bar{\bar{f}}$ is a quasi-isomorphism, then $\bar{f}$ is also a quasi-isomorphism. Now apply Proposition \ref{quasionfactors} again, to have that if $\bar{f}$ is a quasi-isomorphism, then $f$ is also a quasi-isomorphism. To prove the stabilization invariance, we need to verify that $f=D$ is a quasi-isomorphism. For this, it is enough to show that $\bar{\bar{f}}$ is a quasi-isomorphism between $C_3$ and $D_3$.

\begin{defn}
The \emph{fully blocked grid chain complex} $\widetilde{GC}(\GG)$ associated to the grid diagram $\GG$ is a free $\mathbb{F}_2$-module generated by the grid states of $\GG$ with the differential
$$\tilde{\partial}_{\OO, \XX}(\gs x)=\sum\limits_{\gs y \gs S(\GG)}\sum\limits_{\{r\in \rect(\gs x, \gs y)| r\cap \OO=r \cap \XX=\emptyset\}}\gs y.$$
\end{defn}

\begin{prop}
$\bar{\bar{f}}$ is a quasi-isomorphism.
\end{prop}
\begin{proof}
Recall the notion of the fully blocked grid homology (\cite{ozsvath2015grid}), which is the simplest version of grid homology. 

Observe that $C_3\cong \widetilde{GC}(\GG')$ and $D_3 \cong \widetilde{GC}(\GG)\oplus \widetilde{GC}(\GG)_{\shift{1-t}}$. From the proof of \cite[ Proposition 5.2.2]{ozsvath2015grid}, it follows that there is a quasi-isomorphism from $C_3$ to $D_3$, which is in fact $\bar{\bar{f}}$.
\end{proof}

\section{The knot invariant $\Upsilon$}\label{upsilonchapter}
\subsection{The definition of $\Upsilon$}

Let $M$ be a module over $\cR_t$. The torsion submodule $\Tors(M)$ of $M$ is
$$\Tors(M)=\{m\in M \ |\ \mathrm{\ there\ is\ a\ non-zero\ } p \in \cR_t \mathrm{\ with\ } p\cdot m=0\}.$$

\begin{defn}\label{def:Upsilon}
For $t\in [1,2]$,
$$\Upsilon_\mathbb{G}(t):=\max \{\gr(x) | x\in tGH^-(\mathbb{G}),\ x\ \mathrm{is\ homogeneous\ and\ non-torsion}\}$$
\end{defn}

\begin{thm}\label{upsiloninvariantundergridmove}
Let $\GG$ and $\GG'$ be two grid diagrams such that $\GG'$ can be obtained from $\GG$ by a grid move. If $t\in [1,2]$, then $\Upsilon_\mathbb{G}(t)=\Upsilon_\mathbb{G'}(t)$.
\end{thm}
\begin{proof}
If $\GG'$ can be obtained from $\GG$ by a commutation, then from Theorem \ref{thm:commutationinvariance} we know that $tGH^-$ is invariant under commutation, thus $\Upsilon_\mathbb{G}(t)=\Upsilon_\mathbb{G'}(t)$.

Suppose that $\GG'$ can be obtained from $\GG$ by a stabilization, and let $d$ denote the maximal grade which appears among homogeneous non-torsion elements of $tGH^-(\GG)$. Then the maximal grade of homogeneous non-torsion elements of $tGH^-(\GG)_{\shift{1-t}}$ is $d+1-t$. In case of $t\in [1,2]$, $d+1-t\leq d$, thus the maximal grade of homogeneous non-torsion elements of $tGH^-(\GG') \cong tGH^-(\GG)\oplus tGH^-(\GG)_{\shift{1-t}}$ is also $d$.
\end{proof}

In case of $t\in [0,1]$ the above proof does not work, but the $\Upsilon$ invariant can be defined by extending it to $t\in [0,2]$ symmetrically. For $t\in [0,1]$, let $\Upsilon_\mathbb{G}(t)=\Upsilon_\mathbb{G}(2-t)$. Obviously, $\Upsilon_\mathbb{G}(t)$ is also a knot invariant for $t\in [0,1]$.

Consequently, for a knot $K$ and $t\in [0,2]$ we can define $\Upsilon_K(t)$ by taking a grid diagram $G$ representing $K$:
$$\Upsilon_K(t):=\Upsilon_\mathbb{G}(t).$$
According to Theorem \ref{upsiloninvariantundergridmove}, this is well-defined, and the following corollary is immediate:

\begin{cor}
For $t\in [1,2]$, $\Upsilon_K(t)$ is a knot invariant, that is, if $K_1$ and $K_2$ are isotopic knots, then $\Upsilon_{K_1}(t)=\Upsilon_{K_2}(t)$.
\end{cor}

\subsection{Crossing changes}

In this section we examine how $\Upsilon$ changes under crossing changes.

Let $K_+$ and $K_-$ be two knots which differ only in a crossing.

\begin{prop}\label{crossingchangemaps}
There are $\cR_t$-module maps
$$C_-: tGH^-(K_+)\rightarrow tGH^-(K_-) \quad \mathrm{and} \quad C_+: tGH^-(K_-)\rightarrow tGH^-(K_+)$$
where $C_-$ is homogeneous and preserves the grading, and $C_+$ is homogeneous of degree $-2+t$.
$C_-\circ C_+$ and $C_+\circ C_-$ are both the multiplication by $U^{2-t}$.
\end{prop}

Before proving this proposition we introduce some notations:
Represent the knots $K_+$ and $K_-$ by the grid diagrams $\GG_+$ and $\GG_-$ differing by a cross-commutation of columns. Again, we draw these two diagrams onto the same torus so that the $X$- and the $O$-markings are fixed. Using the same notations as in Section \ref{commutation_invariance}, let $\alpha=\{\alpha_1,...,\alpha_n\}$ be the horizontal circles of the diagrams, $\beta=\{\beta_1,...,\beta_n\}$ the vertical circles of $\GG_+$, and $\gamma=\{\beta_1,...,\beta_{i-1}, \gamma_i, \beta_{i+1},...\beta_n\}$ the vertical circles of $\GG_-$. Draw $\beta_i$ and $\gamma_i$ so that they meet perpendicularly in four points, and these intersection points do not lie on any of the horizontal circles. This way the complement of $\beta_i \cup \gamma_i$ has five components, four of which are bigons marked by a single $X$ or a single $O$. Label the $O$-markings so that $O_1$ is above $O_2$. Now the bigon marked by $O_2$ has a common vertex with one of the $X$-marked bigons; denote this point by $s$. The two $X$-labelled bigons share a vertex on $\beta_i\cap \gamma_i$, since $\GG_+$ can be obtained from $\GG_-$ by a cross-commutation. Call this common point $t$ (Figure \ref{fig:crossingchange_sametorus}). 

\begin{figure}[h]
\centering
\includegraphics[width=0.7\textwidth]{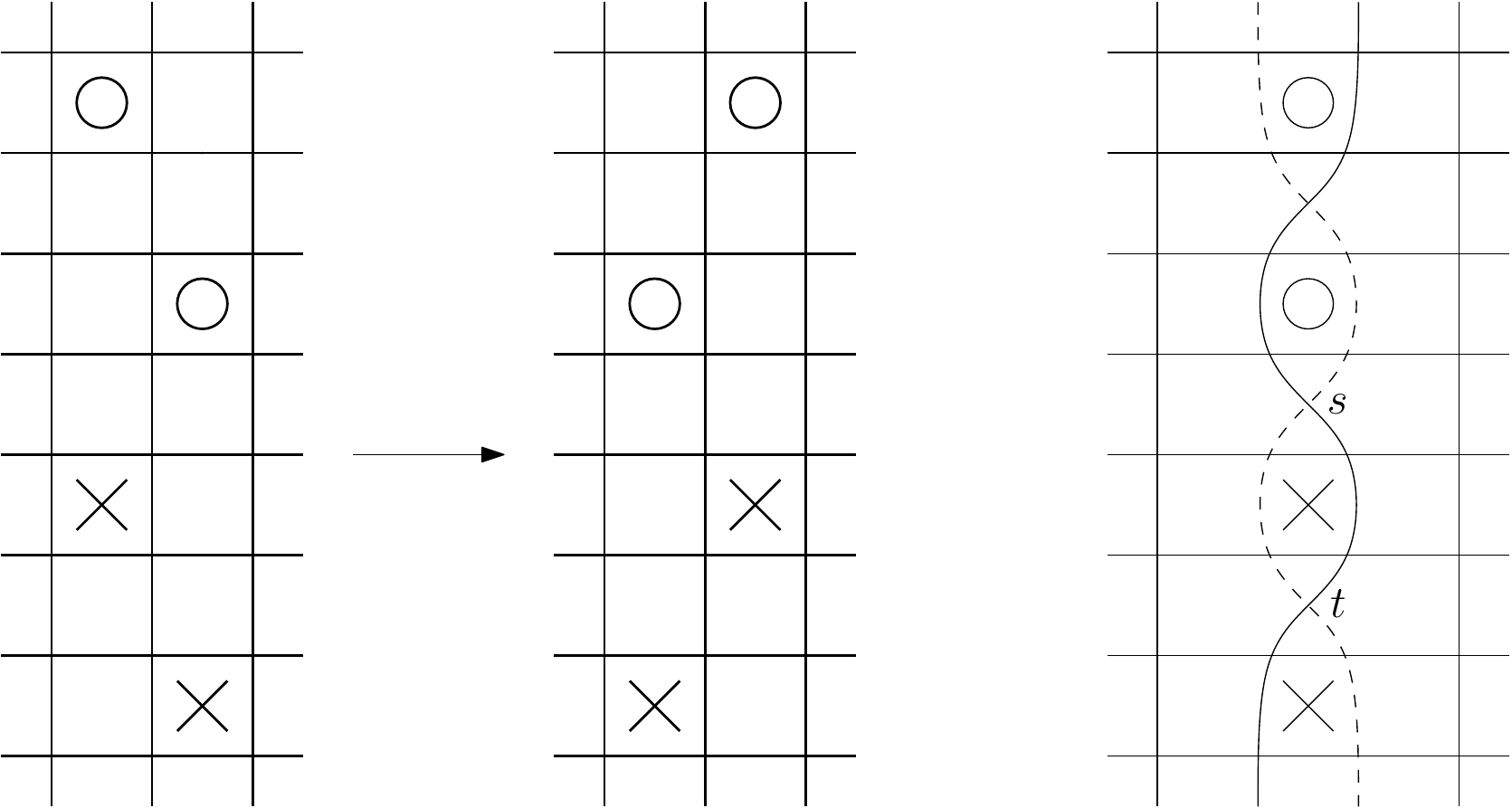}
\caption{Drawing $\GG_+$ and $\GG_-$ on the same torus, and the notation of the vertices}
\label{fig:crossingchange_sametorus}
\end{figure}

\begin{proof}[Proof of Proposition \ref{crossingchangemaps}]
Fix grid states $\gs x_+\in \gs S(\GG_+)$ and $\gs x_-\in \gs S(\GG_-)$. We use the notation $\pent_s(\gs x_+, \gs x_-)$ for the set of empty pentagons from $\gs x_+$ to $\gs x_-$ containing $s$ as a vertex, and similarly, $\pent_t(\gs x_-, \gs x_+)$ for the set of empty pentagons from $\gs x_-$ to $\gs x_+$ with one vertex at $t$. 

Consider the $\cR_t$-module maps $c_-:tGC^-(\GG_+)\rightarrow tGC^-(\GG_-)$ and \mbox{$c_+:tGC^-(\GG_-)\rightarrow tGC^-(\GG_+)$} defined on a grid state $\gs x_+\in \gs S(\GG_+)$ and $\gs x_-\in \gs S(\GG_-)$ respectively in the following way:
$$
c_-(\gs x_+)=\sum\limits_{\gs y_-\in \gs S(\GG_-)}\  \sum\limits_{p\in \pent_s(\gs x_+, \gs y_-)} U^{t\cdot |\XX\cap p|+(2-t)\cdot |\OO\cap p|} \cdot \gs y_-
$$
$$
c_+(\gs x_-)=\sum\limits_{\gs y_+\in \gs S(\GG_+)}\  \sum\limits_{p\in \pent_t(\gs x_-, \gs y_+)} U^{t\cdot |\XX\cap p|+(2-t)\cdot |\OO\cap p|} \cdot \gs y_+.
$$

\begin{lemma}\label{c-c+grading}
The map $c_-$ preserves the $t$-grading, while $c_+$ drops the $t$-grading by $2-t$.
\end{lemma}
\begin{proof}
The four marking between the circles $\beta_i$ and $\gamma_i$ divide $\beta_i$ into four segments, call them $\mathbf{A}$, $\mathbf{B}$, $\mathbf{C}$ and $\mathbf{D}$ according to the followings: let $\mathbf{A}$ be the segment between the two $O$-markings, and from this part on, the order of the other segments to south is $\mathbf{B}$, $\mathbf{C}$ and $\mathbf{D}$. There is a natural one-to-one correspondence between $\gs S(\GG_+)$ and $\gs S(\GG_-)$, assigning to $\gs x_+ \in \gs S(\GG_+)$ the grid state $\gs x_- \in \gs S(\GG_-)$ that agrees with $\gs x_+$ in $n-1$ elements. It is not hard to compute the difference between the grading of $\gs x_-$ and $\gs x_+$:
$$\gr(\gs x_-)=\gr(\gs x_+)+\left\{\begin{array}{rcl}
-1+t&\mathrm{if}&\gs x_+\cap \beta_i \in \mathbf{A}\\
1&\mathrm{if}&\gs x_+\cap \beta_i \in \mathbf{B}\\
1-t&\mathrm{if}&\gs x_+\cap \beta_i \in \mathbf{C}\\
1&\mathrm{if}&\gs x_+\cap \beta_i \in \mathbf{D}.
\end{array}\right.$$

We partition pentagons into left and right ones depending on whether they lie on the left or on the right side of the cross-commutation. Associate to each pentagon $p\in \pent(\gs x_+, \gs y_-)$ the rectangle $r\in \rect(\gs x_+, \gs y_+)$ with the same local multiplicities of the small cells as $p$, except from the four bigons between $\beta_i$ and $\gamma_i$. From Proposition \ref{delta_grading} we know that 
$$\gr (\gs x_+)=\gr (\gs y_+)-t\cdot |\XX\cap r|-(2-t)\cdot |\OO\cap r|+1.$$

Consider the left pentagons. We will do the computations by distinguishing four cases depending on the position of $\gs y_+\cap \beta_i$:
\begin{itemize}
\item If $\gs y_+ \cap \beta_i \in \mathbf{A}$, then
$$\gr(\gs x_+)=\gr(\gs y_+)-t\cdot |\XX\cap r|-(2-t)\cdot |\OO\cap r|+1=$$
$$=\gr(\gs y_-)+1-t-t\cdot |\XX\cap p|-(2-t)\cdot (|\OO\cap p|+1)+1=\gr (U^{t\cdot |\XX\cap p|+(2-t)\cdot |\OO\cap p|}\cdot \gs y_-).$$

\item If $\gs y_+ \cap \beta_i \in \mathbf{B}$, then
$$\gr(\gs x_+)=\gr(\gs y_+)-t\cdot |\XX\cap r|-(2-t)\cdot |\OO\cap r|+1=$$
$$=\gr(\gs y_-)-1-t\cdot |\XX\cap p|-(2-t)\cdot |\OO\cap p|+1=\gr (U^{t\cdot |\XX\cap p|+(2-t)\cdot |\OO\cap p|}\cdot \gs y_-).$$

\item If $\gs y_+ \cap \beta_i \in \mathbf{C}$, then
$$\gr(\gs x_+)=\gr(\gs y_+)-t\cdot |\XX\cap r|-(2-t)\cdot |\OO\cap r|+1=$$
$$=\gr(\gs y_-)-1+t-t\cdot (|\XX\cap p|+1)-(2-t)\cdot |\OO\cap p|+1=\gr (U^{t\cdot |\XX\cap p|+(2-t)\cdot |\OO\cap p|}\cdot \gs y_-).$$

\item If $\gs y_+ \cap \beta_i \in \mathbf{D}$, then
$$\gr(\gs x_+)=\gr(\gs y_+)-t\cdot |\XX\cap r|-(2-t)\cdot |\OO\cap r|+1=$$
$$=\gr(\gs y_-)-1-t\cdot |\XX\cap p|-(2-t)\cdot |\OO\cap p|+1=\gr (U^{t\cdot |\XX\cap p|+(2-t)\cdot |\OO\cap p|}\cdot \gs y_-).$$
\end{itemize}

For the right pentagons and for the map $c_+$ the computation goes the same way.
\end{proof}

\begin{lemma}\label{c_c+chain_maps}
The maps $c_-$ and $c_+$ are chain maps.
\end{lemma}
\begin{proof}
The proof is similar to the reasoning of Lemma \ref{Pchainmap}. If we consider the expression $\partial^-_t\circ c_-(\gs x_+)+c_-\circ \partial^-_t(\gs x_+)$, then most of the domains contribute in pairs.

However, there might be two exceptional ones that admit a unique decomposition. These exceptional domains connect grid states $\gs x_+\in \GG_+$ and $\gs x_-\in \GG_-$ that agree in all but one component. There are two thin annular regions $A_1$ and $A_2$ that have exactly three corners: one vertex is at $s$, and the other two are the components which distinguish $\gs x_+$ and $\gs x_-$ (see Figure 6.4 and 6.5 in \cite{ozsvath2015grid}). Both $A_1$ and $A_2$ have a unique decomposition as either a juxtaposition of an empty pentagon with a vertex at $s$ followed by an empty rectangle in $\GG_-$ or as a juxtaposition of an empty rectangle in $\GG_+$ followed by an empty pentagon with a vertex at $s$. Since $A_1$ and $A_2$ contains exactly the same $X$- and $O$-markings, their contributions to $\partial^-_t\circ c_-(\gs x_+)+c_-\circ \partial^-_t(\gs x_+)$ cancel. Therefore $\partial^-_t\circ c_-(\gs x_+)+c_-\circ \partial^-_t(\gs x_+)=0$.

The same argument shows that $c_+$ is also a chain map.
\end{proof}

The above chain maps $c_-$ and $c_+$ induce the desired maps $C_-$ and $C_+$ on the homologies. In order to verify Proposition \ref{crossingchangemaps}, we have to show that $C_-\circ C_+$ and $C_+\circ C_-$ are both the multiplication by $U^{2-t}$. For this aim, we construct chain homotopies between the composites $c_-\circ c_+$, respectively $c_+\circ c_-$, and multiplication by $U^{2-t}$. 

For $\gs x_-, \gs y_-\in \gs S(\GG_-)$, let $\hex_{s,t}(\gs x_-, \gs y_-)$ denote the set of empty hexagons with two consecutive corners at $s$ and at $t$ in the order consistent with the orientation of the hexagon. The set $\hex_{t,s}$ for $\gs x_+, \gs y_+\in \gs S(\GG_+)$ is defined analogously.

Let $H_-:tGC^-(\GG_-)\rightarrow tGC^-(\GG_-)$ be the $\cR_t$-module map whose value on any $\gs x_-\in \gs S(\GG_-)$ is 
$$H_-(\gs x_-)= \sum\limits_{\gs y_-\in \gs S(\GG_-)}\ \sum\limits_{h\in \hex_{s,t}(\gs x_-, \gs y_-)} U^{t\cdot |\XX\cap h|+(2-t)\cdot |\OO\cap h|}\cdot \gs y_-.$$
The analogous map $H_+:tGC^-(\GG_+)\rightarrow tGC^-(\GG_+)$ is defined in the same way using $\hex_{t,s}(\gs x_+, \gs y_+)$.

It can be shown that $H_-$ and $H_+$ drop the $t$-grading by $1-t$: if a hexagon $h$ from $\gs x_+$ to $\gs y_+$ is counted in $H_+$, then there exists a corresponding empty rectangle $r$ from $\gs x_+$ to $\gs y_+$ that contains one more $X$-marking, and the same number of $O$-markings as $h$. 

Following the lines of the proof of Lemma \ref{H_chain_homotopy}, we can easily verify that $H_+$ is a chain homotopy between $c_+\circ c_-$ and the multiplication by $U^{2-t}$, and that $H_-$ is a chain homotopy between $c_-\circ c_+$ and $U^{2-t}$, i.e.:
$$\partial^-_t\circ H_+ + H_+\circ \partial^-_t=c_+\circ c_-+U^{2-t},$$
$$\partial^-_t\circ H_- + H_-\circ \partial^-_t=c_-\circ c_++U^{2-t}.$$
Therefore, we have that $C_-\circ C_+$ and $C_+\circ C_-$ are both the multiplication by $U^{2-t}$, which completes the proof of Proposition \ref{crossingchangemaps}.
\end{proof}

As a corollary of this proposition, we can determine the rank of $t$-modified grid homology (recall that the rank of an $\cR_t$-module $M$ is the number of the free summands in $\faktor{M}{\Tors(M)}$).


\begin{lemma}\label{torsinjective}
Let $M$ and $N$ be two modules over $\cR_t$. If $\varphi: M\rightarrow N$ and $\psi: N\rightarrow M$ are two module maps with the property that $\psi \circ \varphi$ is the multiplication by $U^{\alpha}$ for some $\alpha\in \RR_{\geq 0}$, then $\varphi$ induces an injective map from $\faktor{M}{\Tors(M)}$ into $\faktor{N}{Tors(N)}$.
\end{lemma}
\begin{proof}
Let $\gs x\in \Tors(M)$ be a torsion element, i.e. there exist an $r\in \mathcal{R}$ such that $r \cdot \gs x =0$. Then $\varphi(\gs x)$ is also a torsion element, since $r \cdot \varphi(\gs x)=\varphi(r \cdot \gs x)=0$. Thus $\varphi$ maps torsion elements to torsion elements and the same holds for $\psi$. So $\varphi$ and $\psi$ indeed induce well-defined homomorphisms between $\faktor{M}{\Tors(M)}$ and $\faktor{N}{Tors(N)}$. 

For the injectivity it is enough to show that $\varphi$ and $\psi$ maps non-torsion elements to non-torsion elements. Suppose that for $\gs x\in M$ we have that $\varphi(\gs x)\in \Tors(N)$, that is, there exists an $s\in \mathcal{R}$ such that $s \cdot \varphi(\gs x)=0$. Then $\psi(s \cdot\varphi( \gs x))=s \cdot \psi(\varphi(\gs x))=s \cdot U^t\cdot \gs x=0$, so $\gs x\in \Tors(M)$, which means that $\varphi$ maps non-torsion elements to non-torsion elements.
\end{proof}

\begin{lemma}
For any grid diagram $\GG$ with grid number $n\geq 2$, $$\faktor{tGH^-(\GG)}{Tors(tGH^-(\GG))}\cong \cR_t^{2^{n-1}}.$$ 
\end{lemma}
\begin{proof}
For any $n\geq 2$ we can draw an $n \times n$ grid diagram $\GG_n$ for the unknot, such that the $X$-markings are in the main diagonal of $\GG_n$, and the $O$-markings are the eastern neighbours of the $X$-markings. With a straightforward calculation it can be shown that $tGH^-(\GG_2)\cong \cR_t^2$. Note that the diagram $\GG_{n+1}$ can be obtained from $\GG_n$ by a stabilization. Therefore the rank of $tGH^-(\GG_{n+1})$ is twice the rank of $tGH^-(\GG_n)$. Hence $\faktor{tGH^-(\GG_n)}{\Tors(tGH^-(\GG_n))}\cong \cR_t^{2^{n-1}}$.

Since any grid diagram $\GG$ with grid number $n$ can be connected to $\GG_n$ by a sequence of crossing changes, from Proposition \ref{crossingchangemaps} and Lemma \ref{torsinjective} we get an injective module map from $\faktor{tGH^-(\GG)}{\Tors(tGH^-(\GG))}$ to $\cR_t^{2^{n-1}}$. As every submodule of $\cR_t^m$ is of the form $\cR_t^k$, where $k\leq m$, we have that $\faktor{tGH^-(\GG)}{\Tors(tGH^-(\GG)}\cong \cR_t^r$, where $r\leq 2^{n-1}$.\\
Proposition \ref{crossingchangemaps} and Lemma \ref{torsinjective} also give an inclusion of $\cR_t^{2^{n-1}}$ into $\faktor{tGH^-(\GG)}{\Tors(tGH^-(\GG))}$, from which $r=2^{n-1}$ follows.
\end{proof}

Another corollary of Proposition \ref{crossingchangemaps} is that we can give a bound for the change of $\Upsilon$ under crossing changes:

\begin{thm}\label{thm:crossingchangedifference}
If the knots $K_+$ and $K_-$ differ in a crossing change, then for $t\in [1,2]$:
$$\Upsilon_{K_+}(t)\leq \Upsilon_{K_-}(t)\leq \Upsilon_{K_+}(t)+(2-t),$$
and from the symmetry of $\Upsilon$, for $t\in [0,1]$:
$$\Upsilon_{K_+}(t)\leq \Upsilon_{K_-}(t)\leq \Upsilon_{K_+}(t)+t.$$
\end{thm}
\begin{proof}
Consider a non-torsion element $\xi\in tGH^-(K_-)$ that has grading $\Upsilon_{K_-}(t)$.\\
As we have seen in the proof of Lemma \ref{torsinjective}, according to Proposition \ref{crossingchangemaps}, $C_+(\xi)$ is non-torsion, and its grading is $\Upsilon_{K_-}(t)-2+t$. Thus $\Upsilon_{K_-}(t)\leq \Upsilon_{K_+}(t)+2-t$.\\
Similarly, if $\eta\in tGH^-(K_+)$ is a non-torsion element with grading $\Upsilon_{K_+}(t)$, then its image $C_-(\eta)$ has grading $\Upsilon_{K_+}(t)$ as well. Since $C_-(\eta)$ is non-torsion, $\Upsilon_{K_+}(t) \leq \Upsilon_{K_-}(t)$, concluding the proof.
\end{proof}

Theorem \ref{thm:crossingchangedifference} implies a bound on the unknotting number $u(K)$ of the knot $K$.
\begin{cor}\label{cor:unknottingbound}
For $t\in [1,2]:$
$$|\Upsilon_K(t)|\leq u(K)\cdot (2-t),$$
while for $t \in [0,1]:$
$$|\Upsilon_K(t)|\leq u(K)\cdot t$$
\end{cor}

\begin{proof}
Recall that for the unknot $\Upsilon_U(t)=0$. From Theorem \ref{thm:crossingchangedifference} we know that in case $t\in [1,2]$ and $t\in [0,1]$, a crossing change alters the absolute value of $\Upsilon$ by $2-t$ and $t$, respectively. By definition, the unknotting number of $K$ is the minimum number of crossing changes needed to transform $K$ to the trivial knot, which easily leads to the desired inequality. For the bound in the case $\upsilon(K)=\Upsilon_K(1)$ see \cite[Remark 3.4.]{unoriented}.
\end{proof}

\bibliographystyle{amsplain}
\bibliography{biblio}

\providecommand{\bysame}{\leavevmode\hbox to3em{\hrulefill}\thinspace}
\providecommand{\MR}{\relax\ifhmode\unskip\space\fi MR }
\providecommand{\MRhref}[2]{%
  \href{http://www.ams.org/mathscinet-getitem?mr=#1}{#2}
}
\providecommand{\href}[2]{#2}
\begin{thebibliography}{10}

\bibitem{MR1501429}
J.~W. Alexander, \emph{Topological invariants of knots and links}, Trans. Amer.
  Math. Soc. \textbf{30} (1928), no.~2, 275--306. \MR{1501429}

\bibitem{brunn1897}
H.~Brunn, \emph{Über verknotete kurven}, Verhandlungen des Internationalen
  Math. Kongresses (Zurich 1897) (1898), 256--259.

\bibitem{Cromwell199537}
P.~R. Cromwell, \emph{Embedding knots and links in an open book i: Basic
  properties}, Topology and its Applications \textbf{64} (1995), no.~1, 37 --
  58.

\bibitem{MR2232855}
I.~A. Dynnikov, \emph{Arc-presentations of links: monotonic simplification},
  Fund. Math. \textbf{190} (2006), 29--76. \MR{2232855}

\bibitem{MR555835}
H.~C. Lyon, \emph{Torus knots in the complements of links and surfaces},
  Michigan Math. J. \textbf{27} (1980), no.~1, 39--46. \MR{555835}

\bibitem{MR2480614}
C.~Manolescu, P.~Ozsv{\'a}th, and S.~Sarkar, \emph{A combinatorial description
  of knot {F}loer homology}, Ann. of Math. (2) \textbf{169} (2009), no.~2,
  633--660. \MR{2480614}

\bibitem{MR792695}
L.~P. Neuwirth, \emph{{$\ast$} projections of knots}, Algebraic and
  differential topology---global differential geometry, Teubner-Texte Math.,
  vol.~70, Teubner, Leipzig, 1984, pp.~198--205. \MR{792695}

\bibitem{MR2500576}
L.~Ng and D.~Thurston, \emph{Grid diagrams, braids, and contact geometry},
  Proceedings of {G}\"okova {G}eometry-{T}opology {C}onference 2008, G\"okova
  Geometry/Topology Conference (GGT), G\"okova, 2009, pp.~120--136.
  \MR{2500576}

\bibitem{MR2026543}
P.~Ozsv{\'a}th and Z.~Szab{\'o}, \emph{Knot {F}loer homology and the four-ball
  genus}, Geom. Topol. \textbf{7} (2003), 615--639. \MR{2026543}

\bibitem{ozsvath2015grid}
P.~S. Ozsv{\'a}th, A.~I. Stipsicz, and Z.~Szab{\'o}, \emph{Grid homology for
  knots and links:}, Mathematical Surveys and Monographs, American Mathematical
  Society, 2015.

\bibitem{ozsvath:2014}
P.~Ozsváth, A.~Stipsicz, and Z.~Szabó, \emph{Concordance homomorphisms from
  knot floer homology}.

\bibitem{unoriented}
P.~S. Ozsváth, Z.~Szabó, and A.~I. Stipsicz, \emph{{Unoriented Knot Floer
  Homology and the Unoriented Four-Ball Genus}}, International Mathematics
  Research Notices \textbf{2017} (2016), no.~17, 5137--5181.

\bibitem{MR2915478}
S.~Sarkar, \emph{Grid diagrams and the {O}zsv\'ath-{S}zab\'o tau-invariant},
  Math. Res. Lett. \textbf{18} (2011), no.~6, 1239--1257. \MR{2915478}

\end{thebibliography}
\end{document}